\documentclass[a4paper,11pt]{amsart}

\usepackage[utf8]{inputenc}

\usepackage[T1]{fontenc}
\usepackage{lmodern}

\usepackage[foot]{amsaddr}

\usepackage[colorinlistoftodos,bordercolor=orange,backgroundcolor=orange!20,linecolor=orange,textsize=normalsize]{todonotes}

\usepackage{amsmath}  
\usepackage{amssymb}     
\usepackage{bbm}

\usepackage{bm}
\usepackage{hyperref}
\usepackage{mathrsfs}
\usepackage{mathtools} 
\usepackage{paralist}
\usepackage{fixmath}
\usepackage{interval}
\usepackage{nicefrac}
\usepackage[capitalize]{cleveref}

\Crefname{figure}{Figure}{Figures}

\crefformat{equation}{\textup{#2(#1)#3}}
\crefrangeformat{equation}{\textup{#3(#1)#4--#5(#2)#6}}
\crefmultiformat{equation}{\textup{#2(#1)#3}}{ and \textup{#2(#1)#3}}
{, \textup{#2(#1)#3}}{, and \textup{#2(#1)#3}}
\crefrangemultiformat{equation}{\textup{#3(#1)#4--#5(#2)#6}}%
{ and \textup{#3(#1)#4--#5(#2)#6}}{, \textup{#3(#1)#4--#5(#2)#6}}{, and \textup{#3(#1)#4--#5(#2)#6}}

\Crefformat{equation}{#2Equation~\textup{(#1)}#3}
\Crefrangeformat{equation}{Equations~\textup{#3(#1)#4--#5(#2)#6}}
\Crefmultiformat{equation}{Equations~\textup{#2(#1)#3}}{ and \textup{#2(#1)#3}}
{, \textup{#2(#1)#3}}{, and \textup{#2(#1)#3}}
\Crefrangemultiformat{equation}{Equations~\textup{#3(#1)#4--#5(#2)#6}}%
{ and \textup{#3(#1)#4--#5(#2)#6}}{, \textup{#3(#1)#4--#5(#2)#6}}{, and \textup{#3(#1)#4--#5(#2)#6}}

\usepackage{tikz}

\usetikzlibrary{arrows}
\usetikzlibrary{patterns}
\usetikzlibrary{calc}
\usetikzlibrary{shapes}
\usetikzlibrary{positioning}
\usetikzlibrary{math}

\usepackage[scr=boondox,scrscaled=1.05]{mathalfa}

\renewcommand{\geq}{\geqslant}
\renewcommand{\leq}{\leqslant}

\renewcommand{\le}{\leq}
\renewcommand{\ge}{\geq}

\newcommand{\N}{\mathbb{N}} 
\newcommand{\R}{\mathbb{R}}

\newcommand{\card}[1]{|{#1}|}

\newtheorem{theorem}{Theorem}[section]

\newtheorem{lemma}[theorem]{Lemma}

\theoremstyle{definition}
\newtheorem{definition}[theorem]{Definition}

\theoremstyle{remark}

\tikzset{grid/.style={gray!30,very thin}}
\tikzset{axis/.style={gray!50,->,>=stealth'}}
\tikzset{convex/.style={draw=none,fill=lightgray,fill opacity=0.7}}
\tikzset{convexborder/.style={very thick}}
\tikzset{point/.style={blue!50}}
\tikzset{hs/.style={fill opacity=0.3,fill=orange,draw=none}}
\tikzset{hsborder/.style={orange,ultra thick,dashdotted}}

\newcommand{\overbar}[1]{\mkern 1.5mu\overline{\mkern-1.5mu#1\mkern-1.5mu}\mkern 1.5mu}

\newcommand{\conv}{\mathrm{conv}}
\newcommand{\dunion}{\uplus}

\newcommand{\ci}{\mathrm{ci}}
\newcommand{\slope}{\mathrm{sl}}

\title[Convexly independent subsets of Minkowski sums]{Convexly independent subsets of Minkowski sums of convex polygons}
\author[Mateusz Skomra]{Mateusz Skomra$^1$}
\address{$^1$LAAS-CNRS, Universit{\'e} de Toulouse, CNRS, Toulouse, France}
\email{mateusz.skomra@laas.fr}
\author[St\'ephan Thomass\'e]{St\'ephan Thomass\'e$^{2,3}$}
\address{$^2$Univ Lyon, EnsL, UCBL, CNRS,  LIP, F-69342, LYON Cedex 07, France}
\address{$^3$Institut Universitaire de France}
\email{stephan.thomasse@ens-lyon.fr}

\begin{document}

\begin{abstract}We show that there exist convex $n$-gons $P$ and $Q$ such that the largest convex polygon in the Minkowski sum $P+Q$ has size $\Theta(n\log n)$. This matches an upper bound of Tiwary.

\end{abstract}

\maketitle

\section{Introduction}

Let $X$ be a finite set of points in the plane. A subset $C$ of $X$ is \emph{convexly independent} if $C$ forms a convex polygon. We denote by $\ci(X)$ the largest size of a convexly independent subset of $X$. The celebrated Happy Ending Theorem, from Erd\H{o}s and Szekeres~\cite{erdos_szekeres_happy_ending}, asserts that $\ci(X)$ goes to infinity when $|X|$ goes to infinity and $X$ does not have three points on the same line.
More precisely, the minimum value one can achieve for $\ci(X)$ is logarithmic in $|X|$. To the opposite, one can try to maximize $\ci(X)$ when the set $X$ satisfies some geometrical constraint. A well-studied case is when $X$ is the \emph{Minkowski sum} $P+Q\coloneqq\{p+q \colon p\in P,q\in Q\}$ of two sets of points $P$ and $Q$. 

Eisenbrand, Pach, Rothvo{\ss}, and Sopher~\cite{eisenbrand_convex_indep} proved that $\ci(P+Q)=O(n^{4/3})$ when $|P|=|Q|=n$. This result was complemented by a construction of B{\'i}lka, Buchin, Fulek, Kiyomi, Okamoto, Tanigawa, and T{\'o}th~\cite{bilka_convex_indep} showing the existence of such $P$ and $Q$ satisfying $\ci(P+Q)=\Theta(n^{4/3})$. Surprisingly, the set $Q$ they use in the extremal constructions can be chosen convex. A natural question is then to ask for the maximum possible value of $\ci(P+Q)$ when both $P$ and $Q$ are convex polygons. In 2014, Tiwary~\cite{tiwary_convex} proposed an upper bound by showing that $\ci(P+Q)=O\bigl((n+m)\log (n+m)\bigr)$ when $P$ and $Q$ are respectively a convex $n$-gon and a convex $m$-gon. He concluded his paper by mentioning that his upper bound seemed very generous, and left as an open problem the existence of a matching lower bound. Our main result in this paper is that Tiwary's proof indeed provides a sharp bound by exhibiting a matching construction.

\begin{theorem}\label{th:main}
There exist two families $(P_{k})_{k \ge 1}, (Q_{k})_{k \ge 1} \subset \R^{2}$ of convexly independent sets 
such that $\card{P_{k}} = \card{Q_{k}} = 2^{k}$ and $\ci(P_{k} + Q_{k}) \ge (k+2)2^{k-1}$ for all $k \ge 1$.
\end{theorem}

An equivalent point of view of Minkowski sums is to consider $(P+Q)/2$ instead of $P+Q$, and therefore $\ci(P+Q)$ represents the maximum size of a set $M$ of midpoints of $P,Q$-segments in convex position. If we actually draw all segments corresponding to these midpoints, we get a bipartite graph with vertex set $P \dunion Q$, whose edges are all pairs $(p,q) \in P \times Q$ such that $(p+q)/2 \in M$. In our construction, the number of edges of this graph is $cn\log n$, all vertices of $P$ and $Q$ are in convex position, and all midpoints of edges are also in convex position.
This kind of drawing was introduced by Halman, Onn, and Rothblum~\cite{halman_onn_convex_dim} where they define a \emph{strong convex embedding} of a graph $G=(V,E)$ as a function $f \colon V \to \R^{2}$ such that $f(V)$ is convex and $\{(f(x)+f(y))/2 \colon (x,y) \in E \}$ is also convex. They showed that if $G$ admits a strong convex embedding, then $|E|\leq 5n-8$ when $n\geq 3$ is the number of vertices. Recently, Garc\'{\i}a-Marco and Knauer~\cite{garcia-marco_knauer} reduced this bound to $2n-3$. Equivalently their result shows that $\ci(P+P)\leq 2n-3$ when $P$ is a convex $n$-gon. Perhaps surprisingly, our construction shows that when slightly relaxing strong convex embedding to only ask convex positions for the two partite sets of a bipartite graph, the bound goes from linear to $n\log n$.

Another very attractive reason motivating the study of $\ci(P+Q)$ for convex $n$-gons is the famous unit distance problem of Erd\H{o}s and Moser~\cite{erdos_moser_unit} asking for the maximum number of pairs of points in a convex $n$-gon $P$ with distance exactly one. The trick is to observe that if two points $p,p'$ of $P$ have distance $1$, then $p-p'$ lies on the unit circle. A careful counting shows, in particular, that $\ci\bigl(P + (-P) \bigr)$ is an upper bound on the number of unit distance pairs. Unfortunately, our construction shows that $\ci(P+Q)$ cannot be directly used to improve the already known $O(n\log n)$ upper bound on the unit distance problem (see F\"uredi~\cite{furedi_unit_distance} and Aggarwal~\cite{aggarwal_unit} for a sharper estimate). We do not know however if the graphs $G_k$ arising from our construction can be realized as convex unit distance graphs. If not, it would be interesting to find a forbidden pattern in $G_k$. This would extend the list of forbidden matrices provided in~\cite{aggarwal_unit}, as none of them appears in $G_k$.

\section{Proof of the main theorem}

The proof of Tiwary~\cite{tiwary_convex} is based on the fact that any collection of convexly independent points in $\R^2$ can be split into at most four monotone chains as in \cref{fig:polygon_split}. We use the name ``south-east chain'' for the chain that is situated in the bottom-right part of this figure. Our construction is based on this type of chains. More formally, we make the following definition and point out the subsequent lemma.
 
\begin{definition}
A \emph{south-east chain} is a sequence $(a^{(1)}, \dots, a^{(n)}) \subset \R^{2}$ of $n \ge 2$ points in the plane that satisfies the following two conditions. First, the sequence is strictly increasing on both coordinates, i.e., we have $a^{(1)}_{1} < a^{(2)}_{1} < \dots < a^{(n)}_{1}$ and $a^{(1)}_{2} < a^{(2)}_{2} < \dots < a^{(n)}_{2}$ (where $a^{(k)}_{i}$ denotes the $i$th coordinate of the point $a^{(k)}$). Second, the corresponding sequence of consecutive slopes is strictly increasing, i.e., we have
\[
\frac{a^{(2)}_{2} - a^{(1)}_{2}}{a^{(2)}_{1} - a^{(1)}_{1}} < \dots < \frac{a^{(k+1)}_{2} - a^{(k)}_{2}}{a^{(k+1)}_{1} - a^{(k)}_{1}} < \dots < \frac{a^{(n)}_{2} - a^{(n-1)}_{2}}{a^{(n)}_{1} - a^{(n-1)}_{1}} \, .
\]
We say that $n$ is the \emph{length} of a south-east chain $(a^{(1)}, \dots, a^{(n)})$.
\end{definition}

\begin{lemma}\label{le:hull_of_chain}
If the points $(a^{(1)}, \dots, a^{(n)})$ form a south-east chain, then they are convexly independent, i.e., the convex hull of $\{a^{(1)}, \dots, a^{(n)}\}$ has $n$ vertices.
\end{lemma}

\begin{figure}[t]
\centering
\begin{tikzpicture}
\begin{scope}[scale = 0.6]
      \draw[gray!60, ultra thin] (-1.5,-0.5) grid (7.5,7.5);
        \fill[black] (3,0) circle (0.15);
        \fill[black] (4.5,0.5) circle (0.15);
        \fill[black] (6,2) circle (0.15);
       \fill[black] (6.5,4) circle (0.15);
        \fill[black] (5.5,6.2) circle (0.15);
        \fill[black] (3.7,6.6) circle (0.15);
        \fill[black] (1.5,6) circle (0.15);
         \fill[black] (0,4.5) circle (0.15);
          \fill[black] (-0.5,2.5) circle (0.15);
          \fill[black] (0.5,0.6) circle (0.15);
          \draw[thick,dashed] (-1.5,4) -- (7.5,2.7);
          \draw[thick,dashed] (1.5,-0.5) -- (2.7,7.5);

\end{scope}
\end{tikzpicture}
\caption{Splitting convexly independent points into monotone chains.}\label{fig:polygon_split}
\end{figure}

\begin{figure}[t]
\centering
\begin{tikzpicture}
\begin{scope}[scale = 0.6]
      \draw[gray!60, ultra thin] (-2,-1) grid (8,8);
        \draw[very thick] (1,1) -- (5,1.5) -- (7,2);
        \draw[very thick] (-0.37,1.37) -- (1.2,5.08) -- (1.77,7.06);
        \draw[very thick] (0.32,1.18) -- (3.1,3.29) -- (4.38,4.53);
        \fill[black] (1,1) circle (0.15) node[below right]{$a$};
        \fill[black] (5,1.5) circle (0.15) node[below right]{$b$};
        \fill[black] (7,2) circle (0.15) node[below right]{$c$};
        \fill[black] (-0.37,1.37) circle (0.15) node[above left]{$R(a)$};
        \fill[black] (1.2,5.08) circle (0.15) node[above left]{$R(b)$};
        \fill[black] (1.77,7.06) circle (0.15) node[above left]{$R(c)$};
        \fill[black] (4.38,4.53) circle (0.15) node[below right]{$\frac{c+R(c)}{2}$};
        \fill[black] (3.1,3.29) circle (0.15) node[below right]{$\frac{b+R(b)}{2}$};
        \fill[black] (0.32,1.18) circle (0.15) node[below left]{$\frac{a+R(a)}{2}$};
        \fill[black] (0,0) circle (0.15) node[below right]{$(0,0)$};
\end{scope}
\end{tikzpicture}
\caption{Rotation of a flat south-east chain.}\label{fig:rotation}
\end{figure}

Before giving a formal proof of \cref{th:main}, let us explain it intuitively. Our proof is based on some elementary properties of rotations. Let $R \colon \R^{2} \to \R^{2}$ denote the counterclockwise rotation by $60$ degrees centered at zero. Let $a, b, c \in \R^{2}$ be three points on a plane such that $(a,b,c)$ is a south-east chain. Moreover, suppose that this chain is sufficiently flat (more precisely, that the line defined by extending the segment $\interval{b}{c}$ forms an angle smaller than $30$ degrees with the horizontal axis). Then, the images $(R(a), R(b), R(c))$ also form a south-east chain. Moreover, if all the slopes between consecutive points of $(a,b,c)$ are close to $0$, then the corresponding slopes of $(R(a), R(b), R(c))$ are close to $\sqrt{3}$ (in other words, the corresponding segments form angles close to $60$ degrees with the horizontal axis). Even more, under these conditions, the triple 
$\frac{1}{2}(a + R(a), b + R(b), c + R(c))$ also forms a south-east chain, and its slopes are close to $\nicefrac{1}{\sqrt{3}}$ (which corresponds to the angle of $30$ degrees). We refer to \cref{fig:rotation} for an illustration. The same applies to south-east chains formed by more than three points. In this way, given a sufficiently flat south-east chain, we can construct two new chains, one with slopes close to $\sqrt{3}$ and one with slopes close to $\nicefrac{1}{\sqrt{3}}$.

Consider now the linear map $L_{\varepsilon} \colon \R^{2} \to \R^{2}$, $L_{\varepsilon}(x,y) \coloneqq (\varepsilon x, \varepsilon^{2} y)$, where $\varepsilon > 0$. It is immediate to see that this map preserves south-east chains (i.e., an image of a south-east chain under $L_{\varepsilon}$ is again a south-east chain). Moreover, this map flattens the chains. In other words, if $A$ is a south-east chain and $\varepsilon > 0$ is small enough, then the image of $A$ under $L_{\varepsilon}$ is a south-east chain that is sufficiently flat to apply the previous observations. Moreover, for small $\varepsilon$, the image of $A$ under $L_{\varepsilon}$ is contained in a small neighborhood of $0$.

\begin{figure}[t]
\centering
\begin{tikzpicture}[xscale=3.5,yscale=3.5, min/.style={draw,fill=black, scale=0.5, shape=circle}]
      \draw[gray!60, ultra thin, xstep =0.25, ystep=0.25] (-0.1,-0.1) grid (1.1,2.6);
      \tikzmath{\x1 = 0.1; \x2 = 0.07; \x3 = 0.3; \y1 = 0; \y2 = \x2*sqrt(3); \y3 = \x3/sqrt(3);}
      \node[min] (A) at (0,0){};
      \draw (0,0) node[below right]{$A$};
      \draw[thick] (0 - \x1, 0 - \y1) -- (0 + \x1, 0 + \y1);
      \node[min] (B) at (0,2){};
      \draw (0,2) node[below right]{$B$};
      \draw[thick] (0 - \x1, 2 - \y1) -- (0 + \x1, 2 + \y1);
      \node[min] (C) at (0,1){};
      \draw (0,1) node[below right]{$C$};
      \draw[thick] (0 - \x1, 1 - \y1) -- (0 + \x1, 1 + \y1);
      \node[min] (A') at (1,5/2){};
      \draw (1,5/2) node[below right]{$A'$};
      \draw[thick] (1 - \x2, 5/2 - \y2) -- (1 + \x2, 5/2 + \y2);
      \node[min] (B') at (1,1){};
      \draw (1,1) node[below right]{$B'$};
      \draw[thick] (1 - \x2, 1 - \y2) -- (1 + \x2, 1 + \y2);
      \node[min] (C') at (1,7/4){};
      \draw (1,7/4) node[below right]{$C'$};
      \draw[thick] (1 - \x2, 7/4 - \y2) -- (1 + \x2, 7/4 + \y2);
      
      \node[min] (D) at (1/2,5/4){};
      \draw (1/2,5/4) node[below right]{$D$};
      \draw[thick] (1/2 - \x3, 5/4 - \y3) -- (1/2 + \x3, 5/4 + \y3);
      
      \draw[thick,dashed] (C) -- (D);
      \draw[thick,dashed] (D) -- (C');
      \draw[thick,dashed] (A) -- (B');
      \draw[thick,dashed] (B) -- (A');
                        
      \end{tikzpicture}
\caption{The construction of south-east chains.}\label{fig:main}
\end{figure}

Suppose now that we are given three south-east chains $A, B, C$ such that $C$ is included in the set of points $(A  + B)/2$. By applying $L_{\varepsilon}$ to all three chains, we can suppose that they are arbitrarily flat, and contained in a small neighborhood of $0$. Then, we can apply the rotation $R$ to $A, B, C$. In this way, we obtain three chains $A', B', C'$ that are again contained in a neighborhood of $0$, but whose slopes are close to $\sqrt{3}$. We now translate the chains as follows. We do not apply any translation to $A$, we translate $B$ by the vector $(0,2)$, and $C$ by the vector $(0,1)$. Then, we translate $A'$ by the vector $(1,\nicefrac{5}{2})$, $B'$ by the vector $(1,1)$, and $C'$ by the vector $(1,\nicefrac{7}{4})$. This gives the situation depicted in \cref{fig:main}. In this picture, the chain $A$ is contained in the small neighborhood of the point marked by $A$, and the slopes of this chain are close to the slope of the solid line passing through $A$. The same is true for the chains $B,C,A',B',C'$. In particular, for sufficiently small $\varepsilon$, the concatenation $\overbar{A} \coloneqq (A, B')$ of chains $A$ and $B'$ forms a south-east chain (because the dashed line from $A$ to $B'$ has slope greater than the slope of the solid line passing through $A$ but smaller than the slope of the solid line passing through $B'$). By the same reasoning, the concatenation $\overbar{B} \coloneqq (B, A')$ forms a south-east chain. Denote $A = (a^{(1)}, \dots, a^{(n)})$ and $A' = (a'^{(1)}, \dots, a'^{(n)})$. By the observation about rotation made above, the sequence $D \coloneqq (\frac{a^{(1)} + a'^{(1)}}{2}, \dots, \frac{a^{(n)} + a'^{(n)}}{2})$ is a south-east chain, and all the slopes of this chain are close to $\nicefrac{1}{\sqrt{3}}$. Moreover, this chain is contained in a small neighborhood of the point $(\nicefrac{1}{2},\nicefrac{5}{4})$. We marked this chain in \cref{fig:main} using the same conventions as for the remaining chains. By applying the same reasoning as above, the concatenation $\overbar{C} \coloneqq (C, D, C')$ is a south-east chain. To summarize, our construction shows the following statement. Given three south-east chains $A, B, C$ such that $C$ is included in the set of points $(A  + B)/2$, we can construct three south-east chains $\overbar{A}, \overbar{B}, \overbar{C}$ such that $\overbar{C}$ is included in $(\overbar{A} + \overbar{B})/2$ and $\card{\overbar{A}} = \card{\overbar{B}} = \card{A} + \card{B}$, $\card{\overbar{C}} = 2\card{C} + \card{A}$. Thus, if we suppose that $\card{A} = \card{B} = n$, then we have $\card{\overbar{A}} = \card{\overbar{B}} = 2n$ and $\card{\overbar{C}} = 2\card{C} + n$. By iterating this reasoning, we obtain the claimed bound $\Theta(n \log n)$.

Before presenting a formal proof, let us discuss the types of graph drawings that we obtain in this way. Here, we are interested in a drawing $f \colon (U \dunion V) \to \R^{2}$ of a bipartite graph $G = (U \dunion V, E)$ such that $f(U)$ is a south-east chain, $f(V)$ is a south-east chain, and the midpoints $\{(f(u) + f(v))/2 \colon (u,v) \in E\}$ also form a south-east chain. The construction described above implies that if $G$ is drawable in this way and $G' \coloneqq (U' \dunion V', E')$ is a copy of $G$, then the graph $\overbar{G} \coloneqq (\overbar{U} \dunion \overbar{V}, \overbar{E})$ defined as $\overbar{U} \coloneqq U \dunion V'$, $\overbar{V} \coloneqq V \dunion U'$, and
\[
\overbar{E} \coloneqq E \dunion E' \dunion \{(u,u') \colon u \in U\}
\]
is also drawable in this fashion. By starting from a graph
\[
G_{1} = \bigl(\{u_{1}, u_{2}\} \dunion \{v_{1}, v_{2}\}, \{(u_{1}, v_{1}), (u_{2}, v_{1}), (u_{2}, v_{2}) \} \bigr) \,
\]
and iterating the procedure, we obtain a family $(G_{k})_{k \ge 1}$ of drawable graphs, each having $2^{k+1}$ vertices and $(k+2)2^{k-1}$ edges. \Cref{fig:graphs} depicts this family of graphs for $k \le 3$ and \cref{fig:g3} depicts a drawing of $G_{3}$ using south-east chains.

\begin{figure}[t]
\begin{minipage}{\linewidth}
\centering
\begin{tikzpicture}[scale = 0.6, min/.style={draw,fill=black, scale=0.5, shape=circle}]
      \node[min] (a) at (0,0){};
      \node[min] (b) at (2,0){};
      \node[min] (c) at (0,2){};
      \node[min] (d) at (2,2){};
      \draw[thick] (a) -- (c); \draw[thick] (b) -- (c); \draw[thick] (b) -- (d);   
      
      \node[min] (a1) at (6,0){};
      \node[min] (b1) at (8,0){};
      \node[min] (c1) at (6,2){};
      \node[min] (d1) at (8,2){};
      \draw[thick] (a1) -- (c1); \draw[thick] (b1) -- (c1); \draw[thick] (b1) -- (d1);    
      \node[min] (a2) at (10,0){};
      \node[min] (b2) at (12,0){};
      \node[min] (c2) at (10,2){};
      \node[min] (d2) at (12,2){};
      \draw[thick] (a2) -- (c2); \draw[thick] (a2) -- (d2); \draw[thick] (b2) -- (d2);
      \draw[thick] (a1) -- (c2); \draw[thick] (b1) -- (d2);
\end{tikzpicture}
\end{minipage}\\[0.6cm]
\begin{minipage}{\linewidth}
\centering
\begin{tikzpicture}[scale = 0.6, min/.style={draw,fill=black, scale=0.5, shape=circle}]
      \node[min] (a1) at (0,0){};
      \node[min] (b1) at (2,0){};
      \node[min] (c1) at (0,2){};
      \node[min] (d1) at (2,2){};
      \draw[thick] (a1) -- (c1); \draw[thick] (b1) -- (c1); \draw[thick] (b1) -- (d1);    
      \node[min] (a2) at (4,0){};
      \node[min] (b2) at (6,0){};
      \node[min] (c2) at (4,2){};
      \node[min] (d2) at (6,2){};
      \draw[thick] (a2) -- (c2); \draw[thick] (a2) -- (d2); \draw[thick] (b2) -- (d2);
      \draw[thick] (a1) -- (c2); \draw[thick] (b1) -- (d2);
      
      \node[min] (a3) at (8,0){};
      \node[min] (b3) at (10,0){};
      \node[min] (c3) at (8,2){};
      \node[min] (d3) at (10,2){};
      \draw[thick] (a3) -- (c3); \draw[thick] (a3) -- (d3); \draw[thick] (b3) -- (d3);    
      \node[min] (a4) at (12,0){};
      \node[min] (b4) at (14,0){};
      \node[min] (c4) at (12,2){};
      \node[min] (d4) at (14,2){};
      \draw[thick] (a4) -- (c4); \draw[thick] (b4) -- (c4); \draw[thick] (b4) -- (d4);
      \draw[thick] (c3) -- (a4); \draw[thick] (d3) -- (b4);
      
      \draw[thick] (a1) -- (c3); \draw[thick] (b1) -- (d3); \draw[thick] (a2) -- (c4); \draw[thick] (b2) -- (d4);
\end{tikzpicture}
\end{minipage}
         
\caption{Graphs that can be drawn using south-east chains.}\label{fig:graphs}
\end{figure}

\begin{figure}[t]
\begin{tikzpicture}[scale = 0.25, min/.style={draw,fill=black, scale=0.5, shape=circle}, max/.style={draw,fill=black, scale=0.5, shape=circle}, av/.style={draw,fill=black, scale=0.35, shape=circle}]

      \node[min] (a1) at (0,0){};
      \node[min] (a2) at (1.25, 0.2){};
      \node[min] (a3) at (13.6, 2.35){};
      \node[min] (a4) at (18.48, 3.54){};
      \node[min] (a5) at (36.93, 8.11){};
      \node[min] (a6) at (39.5, 9.38){};
      \node[min] (a7) at (42.65,11.49){};
      \node[min] (a8) at (43.93,12.51){};
      
      \node[max] (b1) at (4.0, -15.0){};
      \node[max] (b2) at (6.26, -14.8){};
      \node[max] (b3) at (12.69, -13.27){};
      \node[max] (b4) at (16.16, -12.3){};
      \node[max] (b5) at (32.38, -3.99){};
      \node[max] (b6) at (35.8, -2.22){};
      \node[max] (b7) at (41.01, 0.54){};
      \node[max] (b8) at (42.81, 1.81){};
      
      \draw[dashed] (a1) -- (b1); \draw[dashed] (a1) -- (b2); \draw[dashed] (a2) -- (b2); 
      \draw[dashed] (a3) -- (b1); \draw[dashed] (a3) -- (b3); \draw[dashed] (a4) -- (b2); 
      \draw[dashed] (a4) -- (b3); \draw[dashed] (a4) -- (b4); \draw[dashed] (a5) -- (b1);
      \draw[dashed] (a5) -- (b5); \draw[dashed] (a5) -- (b7); \draw[dashed] (a6) -- (b2);
      \draw[dashed] (a6) -- (b5); \draw[dashed] (a6) -- (b6); \draw[dashed] (a6) -- (b8); 
      \draw[dashed] (a7) -- (b3); \draw[dashed] (a7) -- (b7); \draw[dashed] (a7) -- (b8); 
      \draw[dashed] (a8) -- (b4); \draw[dashed] (a8) -- (b8);
      
      \node[av] (c1) at ($(a1) !0.5! (b1)$) {}; 
      \node[av] (c2) at ($(a1) !0.5! (b2)$) {};
      \node[av] (c3) at ($(a2) !0.5! (b2)$) {};
      \node[av] (c4) at ($(a3) !0.5! (b1)$) {};
      \node[av] (c5) at ($(a3) !0.5! (b3)$) {};
      \node[av] (c6) at ($(a4) !0.5! (b2)$) {};
      \node[av] (c7) at ($(a4) !0.5! (b3)$) {};
      \node[av] (c8) at ($(a4) !0.5! (b4)$) {};
      \node[av] (c9) at ($(a5) !0.5! (b1)$) {};
      \node[av] (c10) at ($(a5) !0.5! (b5)$) {};
      \node[av] (c11) at ($(a5) !0.5! (b7)$) {};
      \node[av] (c12) at ($(a6) !0.5! (b2)$) {};
      \node[av] (c13) at ($(a6) !0.5! (b5)$) {};
      \node[av] (c14) at ($(a6) !0.5! (b6)$) {};
      \node[av] (c15) at ($(a6) !0.5! (b8)$) {};
      \node[av] (c16) at ($(a7) !0.5! (b3)$) {};
      \node[av] (c17) at ($(a7) !0.5! (b7)$) {};
      \node[av] (c18) at ($(a7) !0.5! (b8)$) {};
      \node[av] (c19) at ($(a8) !0.5! (b4)$) {};
      \node[av] (c20) at ($(a8) !0.5! (b8)$) {};

\end{tikzpicture}
\caption{Drawing of $G_{3}$ using south-east chains.}\label{fig:g3}
\end{figure}

In the remaining part of this section, we give a formal proof of the argument described above. Let $R \colon \R^{2} \to \R^{2}$ denote the counterclockwise rotation by $60$ degrees centered at zero, i.e., the linear transformation given by the matrix
\[
R \coloneqq \frac{1}{2}\begin{bmatrix}
1 & -\sqrt{3} \\
\sqrt{3} & 1
\end{bmatrix} \, .
\]
Furthermore, if $a \coloneqq (a_{1}, a_{2}), b \coloneqq (b_{1}, b_{2}) \in \R^{2}$ are two points such that $a_{1} < b_{1}$ and $a_{2} < b_{2}$, then we denote by
\[
\slope(a,b) \coloneqq \frac{b_{2} - a_{2}}{b_{1} - a_{1}}
\]
the corresponding slope of the segment $[a, b]$. The following lemma, which can be proven using elementary trigonometric identities, gathers the properties of rotation that were mentioned above.

\begin{lemma}\label{le:rotation}
Suppose that two points $a \coloneqq (a_{1}, a_{2}), b \coloneqq (b_{1}, b_{2}) \in \R^{2}$ are such that $a_{1} < b_{1}$, $a_{2} < b_{2}$, and $\slope(a, b) < \frac{1}{\sqrt{3}} = \tan(\frac{\pi}{6})$. Let $\theta \coloneqq \arctan\bigl(\slope(a,b)\bigr) < \frac{\pi}{6}$ and denote $\tilde{a} = R(a)$, $\tilde{b} = R(b)$. Then, we have $\tilde{a}_{1} < \tilde{b}_{1}$, $\tilde{a}_{2} < \tilde{b}_{2}$,
\[
\slope(\tilde{a}, \tilde{b}) = \tan(\frac{\pi}{3} + \theta)\, , \quad \text{and} \ \quad \slope\Bigl(\frac{a + \tilde{a}}{2}, \frac{b + \tilde{b}}{2}\Bigr) = \tan(\frac{\pi}{6} + \theta) \, .
\]
\end{lemma}
\begin{proof}
We have $\tilde{a} = \frac{1}{2}(a_{1} - \sqrt{3}a_{2}, \sqrt{3}a_{1} + a_{2})$ and $\tilde{b} = \frac{1}{2}(b_{1} - \sqrt{3}b_{2}, \sqrt{3}b_{1} + b_{2})$. The inequality $\tilde{a}_{2} < \tilde{b}_{2}$ is trivial. Moreover, $\tilde{b}_{1} > \tilde{a}_{1} \iff b_{1} - a_{1} > \sqrt{3}(b_{2} - a_{2})$, which is true by our assumptions. Furthermore, we have
\begin{align*}
\tan(\frac{\pi}{3} + \theta) &= \frac{\tan(\frac{\pi}{3}) + \tan(\theta)}{1 - \tan(\frac{\pi}{3})\tan(\theta)}
= \frac{\sqrt{3}(b_{1} - a_{1}) + b_{2} - a_{2}}{b_{1} - a_{1} - \sqrt{3}(b_{2} - a_{2})} \\ 
&= \frac{\tilde{b}_{2} - \tilde{a}_{2}}{\tilde{b}_{1} - \tilde{a}_{1}} = \slope(\tilde{a}, \tilde{b}) \, .
\end{align*}
Similarly,
\begin{align*}
\tan(\frac{\pi}{6} + \theta) &= \frac{\tan(\frac{\pi}{6}) + \tan(\theta)}{1 - \tan(\frac{\pi}{6})\tan(\theta)} = \frac{b_{1}-a_{1} + \sqrt{3}(b_{2} - a_{2})}{\sqrt{3}(b_{1} - a_{1}) - (b_{2} - a_{2})} \\
&= \frac{b_{2} + \tilde{b}_{2} - a_{2} - \tilde{a}_{2}}{b_{1} + \tilde{b}_{1} - a_{1} - \tilde{a}_{1}} = \slope\Bigl(\frac{a + \tilde{a}}{2}, \frac{b + \tilde{b}}{2}\Bigr) . \qedhere
\end{align*}
\end{proof}

For any $\varepsilon > 0$ we denote by $L_{\varepsilon} \colon \R^{2} \to \R^{2}$ the linear transformation $L_{\varepsilon}(x,y) \coloneqq (\varepsilon x, \varepsilon^{2} y)$. As a corollary of \cref{le:rotation} we may now prove the properties of the three transformations of south-east chains discussed above. To improve readability, we use the following notation: if $A$ is a sequence, then we denote by $A(k)$ its $k$th element, so that $A = \bigl(A(1), \dots, A(n)\bigr)$. If $A$ is a south-east chain of length $n$ and $\varepsilon > 0$, then we consider the following three sequences:
\begin{equation}\label{eq:chains}
\begin{aligned}
A_{\varepsilon} &\coloneqq \Bigl( L_{\varepsilon}\bigl(A(1)\bigr), \dots, L_{\varepsilon}\bigl(A(n)\bigr) \Bigr) \, , \\
A'_{\varepsilon} &\coloneqq \Bigl( R\bigl(A_{\varepsilon}(1)\bigr), \dots, R\bigl(A_{\varepsilon}(n)\bigr) \Bigr) \, , \\
A''_{\varepsilon} &\coloneqq \Bigl( \frac{A_{\varepsilon}(1) + A'_{\varepsilon}(1)}{2}, \dots, \frac{A_{\varepsilon}(n) + A'_{\varepsilon}(n)}{2} \Bigr) \, .
\end{aligned}
\end{equation}
Using this notation, $A_{\varepsilon}$ is a chain obtained by flattening $A$, $A'_{\varepsilon}$ is the rotated version of this flattened chain, and $A''_{\varepsilon}$ is the chain formed by taking the midpoints of the two previous chains. The next result follows from \cref{le:rotation}.

\begin{lemma}\label{le:transformations}
Suppose that $A \coloneqq \bigl(A(1), \dots, A(n)\bigr)$ is a south-east chain. Then, for sufficiently small $\varepsilon > 0$, the sequences $A_{\varepsilon}, A'_{\varepsilon}, A''_{\varepsilon}$ are south-east chains. Moreover, for every $k \in [n-1]$ we have the equalities
\begin{align*}
&\lim_{\varepsilon \to 0^{+}} \slope\bigl(A_{\varepsilon}(k),A_{\varepsilon}(k+1)\bigr) = 0 \, , \\
&\lim_{\varepsilon \to 0^{+}} \slope\bigl(A'_{\varepsilon}(k),A'_{\varepsilon}(k+1)\bigr) = \tan(\frac{\pi}{3}) = \sqrt{3} \, , \\
&\lim_{\varepsilon \to 0^{+}} \slope\bigl(A''_{\varepsilon}(k),A''_{\varepsilon}(k+1)\bigr) = \tan(\frac{\pi}{6}) = \frac{1}{\sqrt{3}} \, .
\end{align*}
\end{lemma}
\begin{proof}
It is obvious that the sequence $A_{\varepsilon}$ is strictly increasing on both coordinates. Moreover, for every $k \in [n-1]$ we have $\slope\bigl(A_{\varepsilon}(k), A_{\varepsilon}(k+1)\bigr) = \varepsilon \slope\bigl(A(k), A(k+1)\bigr)$. Hence, $A_{\varepsilon}$ is a south-east chain and we have $\lim_{\varepsilon \to 0^{+}} \slope\bigl(A_{\varepsilon}(k),A_{\varepsilon}(k+1)\bigr) = 0$. To prove the claim for the remaining two sequences, note that for sufficiently small $\varepsilon > 0$, the inequality $\slope \bigl(A_{\varepsilon}(k), A_{\varepsilon}(k+1)\bigr) = \varepsilon \slope\bigl(A(k), A(k+1)\bigr) < \nicefrac{1}{\sqrt{3}}$ is satisfied for all $k \in [n-1]$. Hence, by \cref{le:rotation}, the sequence $A'_{\varepsilon} = \Bigl( R\bigl(A_{\varepsilon}(1)\bigr), \dots, R\bigl(A_{\varepsilon}(n)\bigr) \Bigr)$ is strictly increasing on both coordinates and the same is true for the sequence $A''_{\varepsilon} = \Bigl( \frac{A_{\varepsilon}(1) + A'_{\varepsilon}(1)}{2}, \dots, \frac{A_{\varepsilon}(n) + A'_{\varepsilon}(n)}{2} \Bigr)$. Moreover, if we denote $\theta^{(k,\varepsilon)} \coloneqq \arctan\Bigl( \slope\bigl(A_{\varepsilon}(k), A_{\varepsilon}(k+1)\bigr) \Bigr) < \frac{\pi}{6}$, then the sequence $\theta^{(1,\varepsilon)}, \dots, \theta^{(n-1,\varepsilon)}$ is strictly increasing and \cref{le:rotation} shows that
\[
\slope\bigl( A'_{\varepsilon}(k), A'_{\varepsilon}(k+1) \bigr) = \tan(\frac{\pi}{3} + \theta^{(k,\varepsilon)}) \,
\]
and
\[
\slope\bigl( A''_{\varepsilon}(k), A''_{\varepsilon}(k+1) \bigr) = \tan(\frac{\pi}{6} + \theta^{(k,\varepsilon)}) \, .
\]
In particular, the sequences $A'_{\varepsilon},A''_{\varepsilon}$ are south-east chains. Moreover, we have $\lim_{\varepsilon \to 0^+}\theta^{(k,\varepsilon)} = 0$ for all $k$, which gives the equalities
\begin{align*}
&\lim_{\varepsilon \to 0^{+}} \slope\bigl(A'_{\varepsilon}(k),A'_{\varepsilon}(k+1)\bigr) = \tan(\frac{\pi}{3}) = \sqrt{3} \, , \\
&\lim_{\varepsilon \to 0^{+}} \slope\bigl(A''_{\varepsilon}(k),A''_{\varepsilon}(k+1)\bigr) = \tan(\frac{\pi}{6}) = \frac{1}{\sqrt{3}} \, . \qedhere
\end{align*}
\end{proof}

We now show how the transformations given in \cref{le:transformations} can be used to prove \cref{th:main}. To do so, we take three south-east chains $A \coloneqq \bigl(A(1), \dots, A(n)\bigr)$, $B \coloneqq \bigl(B(1), \dots, B(n)\bigr)$, and $C \coloneqq \bigl(C(1), \dots, C(m)\bigr)$ such that $C \subset (A + B)/2$. As discussed before, we let $u \coloneqq (1,\nicefrac{5}{2})$, $v \coloneqq (0,2)$, $w \coloneqq (1,1)$ and we consider the chains $A_{\varepsilon}, A'_{\varepsilon}, B_{\varepsilon}, B'_{\varepsilon}$ defined as in~\cref{eq:chains}.

\begin{lemma}\label{le:construction}
If $\varepsilon > 0$ is sufficiently small, then the sequences $\overbar{A}_{\varepsilon} \coloneqq (A_{\varepsilon}, w + B'_{\varepsilon})$, $\overbar{B}_{\varepsilon} \coloneqq (v + B_{\varepsilon}, u + A'_{\varepsilon})$ are south-east chains. Moreover, the set $(\overbar{A}_{\varepsilon} + \overbar{B}_{\varepsilon})/2$ contains a south-east chain of length at least $2m + n$.
\end{lemma}
In the statement above, $w + B'_{\varepsilon}$ denotes the sequence obtained by translating every element of $B'_{\varepsilon}$ by the vector $w$ and $(A_{\varepsilon}, w + B'_{\varepsilon})$ denotes the concatenation of two sequences. The same applies to $(v + B_{\varepsilon}, u + A'_{\varepsilon})$.
\begin{proof}[Proof of \cref{le:construction}]
We start by proving that $\overbar{A}_{\varepsilon}$ is a south-east chain. By \cref{le:transformations}, the sequences $A_{\varepsilon}$ and $B'_{\varepsilon}$ are south-east chains for sufficiently small $\varepsilon$. Hence, the sequence $w + B'_{\varepsilon}$ is also a south-east chain. Furthermore, we have $\lim_{\varepsilon \to 0^{+}} \bigl(A_{\varepsilon}(n) \bigr) = (0,0)$ and $\lim_{\varepsilon \to 0^{+}} \bigl(w + B'_{\varepsilon}(1)\bigr) = w = (1,1)$. In particular, for sufficiently small $\varepsilon$, the sequence $\overbar{A}_{\varepsilon}$ is strictly increasing on both coordinates. Moreover, \cref{le:transformations} shows the equalities $\lim_{\varepsilon \to 0^{+}} \slope\bigl( A_{\varepsilon}(n-1), A_{\varepsilon}(n) \bigr) = 0$ and $\lim_{\varepsilon \to 0^{+}} \slope\bigl( w + B'_{\varepsilon}(1), w + B'_{\varepsilon}(2) \bigr) = \sqrt{3}$. We also have
\[
\lim_{\varepsilon \to 0^{+}} \slope\bigl( A_{\varepsilon}(n) , w + B'_{\varepsilon}(1) \bigr) = \slope(0, w) = 1 \, .
\]
Since $0 < 1 < \sqrt{3}$, the sequence $\overbar{A}_{\varepsilon}$ is a south-east chain for sufficiently small $\varepsilon$. The proof for $\overbar{B}_{\varepsilon}$ is analogous---it is enough to observe that $\lim_{\varepsilon \to 0^{+}} \bigl( v + B_{\varepsilon}(n) \bigr) = (0,2)$ and $\lim_{\varepsilon \to 0^{+}} \bigl(u + A'_{\varepsilon}(1)\bigr) = (1,\nicefrac{5}{2})$ to show that $\overbar{B}_{\varepsilon}$ is strictly increasing on both coordinates. Then, $\overbar{B}_{\varepsilon}$ is a south-east chain by the equalities $\lim_{\varepsilon \to 0^{+}} \slope\bigl( v + B_{\varepsilon}(n-1), v + B_{\varepsilon}(n) \bigr) = 0$, $\lim_{\varepsilon \to 0^{+}} \slope\bigl( u + A'_{\varepsilon}(1), u + A'_{\varepsilon}(2) \bigr) = \sqrt{3}$, and 
\[
\lim_{\varepsilon \to 0^{+}} \slope\bigl( v + B_{\varepsilon}(n), u + A'_{\varepsilon}(1) \bigr) = \slope(v, u) = \frac{1}{2} \, .
\]
It remains to show that $(\overbar{A}_{\varepsilon} + \overbar{B}_{\varepsilon})/2$ contains a south-east chain of length at least $2m + n$. To do so, consider the chain $C \subset (A + B)/2$, $|C| = m$, and let $C_{\varepsilon}, C'_{\varepsilon}$ be defined as in~\cref{eq:chains}. Furthermore, let $t \coloneqq (0,1) = v/2$ and $z \coloneqq (1,\nicefrac{7}{4}) = (u + w)/2$. Since the transformations $L_{\varepsilon}$ and $R$ are linear, we have $t + C_{\varepsilon} \subset (A_{\varepsilon} + v + B_{\varepsilon})/2 \subset (\overbar{A}_{\varepsilon} + \overbar{B}_{\varepsilon})/2$ and $z + C'_{\varepsilon} \subset (u + A'_{\varepsilon} + w + B'_{\varepsilon})/2 \subset (\overbar{A}_{\varepsilon} + \overbar{B}_{\varepsilon})/2$. Moreover, we define the chain $A''_{\varepsilon}$ as in~\cref{eq:chains}, we let $s \coloneqq (\nicefrac{1}{2},\nicefrac{5}{4}) = u/2$, and we note that $s + A''_{\varepsilon} \subset (A_{\varepsilon} + u + A'_{\varepsilon})/2 \subset (\overbar{A}_{\varepsilon} + \overbar{B}_{\varepsilon})/2$. Hence, the sequence $D_{\varepsilon} \coloneqq (t + C_{\varepsilon}, s + A''_{\varepsilon}, z + C'_{\varepsilon})$ is contained in $(\overbar{A}_{\varepsilon} + \overbar{B}_{\varepsilon})/2$ and its length is equal to $2m + n$. Therefore, it is enough to prove that $D_{\varepsilon}$ is a south-east chain. The proof is similar to the proofs before. By \cref{le:transformations}, the sequences $t + C_{\varepsilon}$, $s + A''_{\varepsilon}$, and $z + C'_{\varepsilon}$ are south-east chains for sufficiently small $\varepsilon$. Moreover, we have the equalities $\lim_{\varepsilon \to 0^{+}} \bigl( t + C_{\varepsilon}(m) \bigr) = (0,1)$, $\lim_{\varepsilon \to 0^{+}} \bigl( s + A''_{\varepsilon}(1) \bigr) = \lim_{\varepsilon \to 0^{+}} \bigl( s + A''_{\varepsilon}(n) \bigr) = (\nicefrac{1}{2},\nicefrac{5}{4})$, and $\lim_{\varepsilon \to 0^{+}} \bigl( z + C'_{\varepsilon}(1) \bigr) = (1,\nicefrac{7}{4})$, which show that $D_{\varepsilon}$ is strictly increasing on both coordinates for sufficiently small $\varepsilon$. Furthermore, we use \cref{le:transformations} to observe that $\lim_{\varepsilon \to 0^{+}} \slope\bigl( t + C_{\varepsilon}(m-1), t + C_{\varepsilon}(m) \bigr) = 0$, $\lim_{\varepsilon \to 0^{+}} \slope\bigl( s + A''_{\varepsilon}(1), s + A''_{\varepsilon}(2) \bigr) = \lim_{\varepsilon \to 0^{+}} \slope\bigl( s + A''_{\varepsilon}(n-1), s + A''_{\varepsilon}(n) \bigr) = \frac{1}{\sqrt{3}}$, and $\lim_{\varepsilon \to 0^{+}} \slope\bigl( z + C'_{\varepsilon}(1), z + C'_{\varepsilon}(2) \bigr) = \sqrt{3}$. To finish, we have
\begin{align*}
&\lim_{\varepsilon \to 0^{+}} \slope\bigl( t + C_{\varepsilon}(m), s + A''_{\varepsilon}(1) \bigr) = \slope(t,s) = \frac{1}{2} \, , \\
&\lim_{\varepsilon \to 0^{+}} \slope\bigl( s + A''_{\varepsilon}(n), z + C'_{\varepsilon}(1) \bigr) = \slope(s,z) = 1 \, ,
\end{align*}
and $0 < \frac{1}{2} < \frac{1}{\sqrt{3}} < 1 < \sqrt{3}$. Hence, for sufficiently small $\varepsilon$, the sequence $D_{\varepsilon}$ is a south-east chain.
\end{proof}

As a corollary, we may prove our main theorem.

\begin{proof}[Proof of \cref{th:main}]
As noted in the introduction, it is enough to prove the claim for $(P_{k} + Q_{k})/2$ instead of $P_{k} + Q_{k}$. We show, by induction over $k$, that we can find two south-east chains $P_{k}$, $Q_{k}$, each of length $2^{k}$, such that $(P_{k} + Q_{k})/2$ contains a south-east chain of length $(k+2)2^{k-1}$. Let $P_{1} \coloneqq \bigl( (0,0), (2,1)\bigr)$ and $Q_{1} \coloneqq ( (0,2), (2,4) )$. The sequences $P_{1}$ and $Q_{1}$ are south-east chains and the set $(P_{1} + Q_{1})/2$ contains a south-east chain of length three. Moreover, given $P_{k}, Q_{k}$ we may apply \cref{le:construction} to obtain two south-east chains $P_{k+1}, Q_{k+1}$, each of length $2^{k+1}$, such that $(P_{k+1} + Q_{k+1})/2$ contains a south-east chain of length at least $(k+2)2^{k} + 2^{k} = (k+3)2^{k}$. Therefore, the claim follows from \cref{le:hull_of_chain}.
\end{proof}

\section{Final remarks}

Let us discuss some natural questions for further research. Firstly, we do not know how far from optimal is our construction. More precisely, consider the function $f \colon \N^{*} \to \N^{*}$ defined as 
\begin{align*}
f(n) \coloneqq \max\{\ci(P + Q) \colon &P,Q \subset \R^{2} \text{ are convexly independent} \\ &\text{and } \card{P} = \card{Q} = n \} \, .
\end{align*}
By joining our analysis with the result of Tiwary~\cite{tiwary_convex}, we obtain the asymptotic bound $f(n) = \Theta(n \log n)$. One can ask for an optimal constant $C$ such that $f(n) \le C n \log n$. Secondly, we wonder if our family of graphs $G_{k} = (U_{k} \dunion V_{k}, E_{k})$ can be still embedded in the plane if we impose a stronger condition on the midpoints of the edges. For instance, we may ask for an embedding such that $U_{k}$ and $V_{k}$ are south-east chains and the midpoints of $E_{k}$ are contained in some convex curve of small degree, such as a circle or a parabola. Our interest in this question is motivated by the following observation: if $G_{k}$ are realizable in the south-east quadrant of the plane in such a way that the midpoints of $E_{k}$ are on the unit circle, then we obtain a $\Theta(n \log n)$ bound for the convex version of the unit distance problem, proposed by Erd\H{o}s and Moser~\cite{erdos_moser_unit} (see~\cite{furedi_unit_distance,aggarwal_unit} for more information on lower and upper bounds for this variant of the unit distance problem). Indeed, suppose that $h_k \colon (U_k \dunion V_k) \to \R^{2}$ is an embedding of $G_{k}$ such that $h_k(U_k)$ and $h_k(V_k)$ are south-east chains, $h_k(U_k), h_k(V_k) \subset \{x \in \R^2 \colon x_{1} \ge 0, x_{2} \le 0\}$, and $\| \frac{h_k(u) + h_k(v)}{2}\|_2 = 1$ for all $(u,v) \in E_k$. Then, the points $-h_k(U_k)$ form a monotone concave chain in the north-west quadrant of the plane and $h_k(V_k)$ form a monotone convex chain in the south-east quadrant of the plane, which implies that the collection $\{-h_k(U_k), h_k(V_k)\}$ is convexly independent. Therefore, the polygon $P_k \coloneqq \conv\{-h_k(U_k)/2, h_k(V_k)/2\}$ has $2^{k+1}$ vertices and $\Theta(k2^k)$ diagonals of length one. As noted in the introduction, the existence of such an embedding cannot be excluded using the current list of forbidden matrices~\cite{aggarwal_unit}. We were neither able to construct these embeddings nor find a new forbidden pattern that occurs in $G_k$.

\bibliographystyle{acm}

\end{document}